\documentclass[]{article}

\usepackage{amssymb,latexsym,amsmath,graphicx,subfigure,color,amsthm}     
\usepackage{multirow}
\usepackage{epsfig,epstopdf}
\theoremstyle{plain}
\newtheorem{theorem}{Theorem}
\newtheorem{lemma}[theorem]{Lemma}

\theoremstyle{definition}
\newtheorem{assumption}{Assumption}

\title{Some convergence analysis for multicontinuum homogenization}
\author{ Wing Tat Leung\footnote{Department of Mathematics, City University of Hong Kong, Hong Kong}}

\begin{document}
\maketitle
\begin{abstract}
In this paper, we provide an analysis of a recently proposed multicontinuum homogenization technique \cite{efendiev2023multicontinuum}.  The analysis differs from those used in classical homogenization methods for several reasons. First, the cell problems in multicontinuum homogenization use constraint problems and can not be directly substituted into the differential operator. Secondly, the problem contains high contrast that remains in the homogenized problem. The homogenized problem averages the microstructure while containing the small parameter. In this analysis, we first based on our previous techniques, CEM-GMsFEM \cite{chung2018constraint}, to define a CEM-downscaling operator that maps the multicontinuum quantities to an approximated microscopic solution. Following the regularity assumption of the multicontinuum quantities, we construct a downscaling operator and the homogenized multicontinuum equations using the information of linear approximation of the multicontinuum quantities. The error analysis is given by the residual estimate of the homogenized equations and the well-posedness assumption of the homogenized equations.

\end{abstract}

\section{Introduction}
Many challenges exhibit a multiscale nature, such as the flow in porous media with multiscale highly heterogeneous parameter and displaying high contrast. Simulating these complex scenarios typically involves employing a coarse computational grid, where the grid size significantly larger than the scales of heterogeneities. 
These heterogeneities present in the medium at the microscopic scale give rise to diverse properties across different scales which affect the accuracy of the numerical simulation.
Upscaling models play a pivotal role in achieving efficient and accurate numerical simulations fin these cases, thereby enhancing the effectiveness of inversion models, uncertainty quantification models, and optimization models. Moreover, multiscale upscaling models contribute to a comprehensive understanding of physical phenomena, facilitating the establishment of connections between microscopic and macroscopic data. 

Many existing upscaling methods have demonstrated significant success across various applications. The process of obtaining upscaled models can be approached through various methodologies and techniques. For problems exhibiting scale separation with low contrast medium, single continuum upscaling models have matured considerably. In works such as \cite{bourgeat2004approximations,matache2000homogenization,jikov2012homogenization,lipton2006homogenization,engquist2007heterogeneous,engquist2008asymptotic,henning2014localized, jenny2003multi,arbogast2007multiscale,arbogast2002implementation}, various upscaling approaches, including homogenization, numerical homogenization, heterogeneous multiscale methods, and localized orthogonal decomposition (LOD) method, have been introduced. These models effectively address problems in many applications. However, they may exhibit reduced accuracy in scenarios where the multiscale problem lacks scale separation and involves multiple crucial macroscopic quantities, such as problems with high-contrast parameters.

In this paper, we would like to consider a multicontinuum model which can capture the property of the solution with multiple important macroscopic quantities. In the existing literature, several upscaling methods for multicontinuum models are designed to handle non-separable scales. Notable examples include multicontinuum theories \cite{bedford1972multi,fafalis2012capability} and the computational continua framework \cite{fish2010computational}. Additionally, methods like multiscale enrichment \cite{fish2005multiscale}, generalized multiscale finite element method \cite{efendiev2013generalized,efendiev2014generalized}, dual-porosity model \cite{gerke1993dual,arbogast1990derivation}, and others \cite{roberts2017slowly, showalter1991micro,owhadi2007metric,aifantis1979continuum,bunoiu2019upscaling,iecsan1997theory, chai2018efficient,chung2017coupling} stand out as effective approaches for multicontinuum models.
While some of these methods can directly yield a discretized macroscopic model from the microscopic model, this approach may limit the flexibility of numerical discretization choices and introduce indirect connections between models at different scales. Some of these methods may lack a detailed analysis of the upscaling model or focus solely on the dual continuum case, assuming the continua to be fracture and matrix regions. To obtain a multicontinuum upscaling model which can cooperate with any numerical discretization, we propose a multicontinuum homogenization method.

As we mentioned before, homogenization of multiscale problems are widely used in many applications.
One of commonly used homogenization techniques include the asympototic
expansion \cite{bensoussan2011asymptotic,blanc2023homogenization}. 
In these methods, in each representative volume, the effective 
property is computed based on local solutions. However, this approach 
cannot be used for complex heterogeneities and 
approaches based on multicontinuum
ideas are needed \cite{efendiev2023multicontinuum} (and the references therein).

In \cite{efendiev2023multicontinuum}, we propose a general
 approach for multicontinuum
homogenization in the context of numerical homogenization.
The main idea of the method is to represent the solution via its averages
and its average gradients
in each continua. Continua are defined via some characteristic functions
in each RVE (typically related to subregion).
The resulting formal asymptotic expansion entails cell problems 
(defined in each RVE)
and
corresponding the smooth macroscopic functions defined in the whole domain. 
At each macroscale
point, we assign several macroscopic variables. The resulting 
macroscpoopic equations consist of a  diffusion-convection-reaction
system.
The proposed method differs from classical homogenization in several
aspects. First, at each macroscopic point, several macroscale variables
are assigned. Secondly, the cell problems are described via local problems
with constraints in each continua, including the averages of the solution.
These differing aspects make the analysis of the multicontinuum
homogenization approach significantly different.

The analysis of multicontinuum method is based on nonlocal multicontinumm
(NLMC)
approach developed in our previous work \cite{vasilyeva2019nonlocal,chung2018constraint}. Nonlocal multicontinuum approach
is shown to converge independent of the contrast. However, the resulting 
equations appear to be discrete and not in the form of macroscopic laws. 
The goal is to derive
macroscopic conservation laws, which leads to multicontinuum homogenization.
Multicontinuum
homogenization requires smooth macroscopic variables defined globally
and, consequently,  the analysis requires stronger assumptions
compared to nonlocal multicontinuum approach.

 In our analysis, as a first step, we write down
a corrector for multicontinuum homogenization. Secondly,
we  show  the closeness
of multicontinuum homogenization solution to nonlocal multicontinuum
solution, which has similar cell problems. The latter is shown assuming
smooth macroscopic variables. Thirdly, we use the macroscale energy
and compare the multicontinuum homogenized solution corrector to 
the exact solution average used in this corrector. We show that
the error is small if macroscopic functions are smooth. Assuming
that the macroscopic operator is invertible with a bound, we 
can obtain the closeness of the solutions.

We make several simplifying assumptions in the proof, which can be 
removed. For example, we only assume two continua, though the proof
can be extended to an arbitrary number of continua with appropriate assumptions.

This paper is outlined as follows. In the next section, we will discuss the problem setting and some preliminary definition for the latter discussion. In section \ref{sec:homo}, we will introduce our proposed multicontinuum homogenization method. In section \ref{sec:analy}, we will analyze the convergence result of our proposed method.

\section{Preliminaries} \label{sec:Pre}

In this paper, we consider the following partial differential equation
\begin{equation}\label{eq:pde}
    -\nabla\cdot(\kappa_{\epsilon}\nabla u_{\epsilon})=f\;\text{in }\Omega
\end{equation}
with $u_{\epsilon}|_{\partial\Omega}=0$. We assume $\kappa_{\epsilon}$ is a $\epsilon$ depending
heterogeneous permeability with high contrast. That is,
$\cfrac{\max_{x\in\Omega}\kappa_{\epsilon}(x)}{\min_{x\in\Omega}\kappa_{\epsilon}(x)}=\zeta_{\epsilon}\gg1$.
We remark that the discussion in this paper can be extended to a general elliptic differential operator. 
In this paper, we will focus on the weak problem defined as finding $u_{\epsilon}\in H_{0}^{1}(\Omega)$
satisfying 
\[
a_{\epsilon}(u_{\epsilon},v)=(f,v),\;\forall v\in H_{0}^{1}(\Omega),
\]
where $(\cdot,\cdot)$ is the standard $L^{2}$ inner product and $a(\cdot,\cdot)$ is defined as
\[
a_{\epsilon}(v,w)=\int_{\Omega}\kappa_{\epsilon}\nabla v\cdot\nabla w.
\]
In the following discussion, we will extend the domain of the bilinear operator $a_\epsilon(\cdot,\cdot)$ from $H^1(\Omega)\times H^1(\Omega)$ to piece-wise $H^1$ function. For example, if $u|_{\omega_i}$ and $v|_{\omega_i}$ are $H^1$ functions with $\Omega=\omega_1\cup \omega_2$ and $\omega_1\cap \omega_2$, 
\[
a_{\epsilon}(u,v) = \int_{\omega_1}\kappa_{\epsilon}\nabla u|_{\omega_1} \cdot\nabla v|_{\omega_1} + \int_{\omega_2}\kappa_{\epsilon}\nabla u|_{\omega_2}\cdot\nabla v|_{\omega_2}.
\]
For any subdomain $\omega\subset \Omega$, we define norms $\|\cdot \|_{a(\omega)}$ and $\|\cdot\|_a$ by 
\[
\|u\|_{a(\omega)} := \Big(\int_{\omega}\kappa |\nabla u|^2\Big)^{\frac{1}{2}}
\]
and $\|u\|_{a} :=\|u\|_{a(\Omega)}$. 

To simplify the discussion, we consider the domain is a square domain with $\Omega= [0,1]^{d}$ and it can be separated
into two subdomains, $\Omega_{0}$ and $\Omega_{1}=\Omega\backslash\Omega_{0}$.
The subdomains $\Omega_{1}$ and $\Omega_{0}$ represent the region
with high and low permeability with respectively. 
We remark that the discussion can be extended to the cases with multiple continua easily.

Before delving into our proposed homogenization method, we will initially introduce several scales and partitions of domain to our method. This paper considers three distinct scales: the physical microscopic scale $\epsilon$, the "observing" scale $H_\epsilon$, and the computational upscaling scale $H$, with $\epsilon \leq H_\epsilon < H$.
The physical microscopic scale $\epsilon$ and the computational upscaling scale $H$ are fairly straightforward. The physical microscopic scale refers to the scale of the medium's microscopic structure and the differential operator, such as the period of the medium's micro-scale pattern and the medium's high contrast. The computational upscaling scale $H$ can be seen as a macroscopic computational scale for homogenized equations, allowing the homogenized coefficient to be depicted as a piece-wise constant function at this scale.
The observing scale $H_\epsilon$ is the scale at which the multicontinuum quantities are defined. The macroscopic quantities can be perceived as local averages of the multiscale solution in local regions with scale $H_\epsilon$. Choosing different observing scales may yield significant quantities and a homogenized equation of a different nature. For instance, if $H_\epsilon$ is chosen to be much smaller than $\epsilon$, since the multiscale solution varies slowly at this scale, the original equation is already effective and there are no separated multiple continua. If the observing scale is so large that all of the large-scale structures are averaged at this scale, such as when $H_\epsilon$ is larger than the high contrast channel length, the multiple continua feature may also not be observable. Essentially, this observing scale defines the scale for the homogenized equations. We remark that even if there is no multiple continua feature in a certain scale, our proposed multicontinuum upscaling method is still applicable, but the model may force all of the continua to be the same value.
In practice, the observing scale and the computational upscaling scale are selected based on the spatial and temporal scale of interest in the application. It should be noted that continuous homogenized equations with smooth coefficients can be obtained by considering all $\epsilon$, $H_\epsilon$ and $H$ converge to zero in a correct ratio.

We then define a family of partition of the computational domain $\Omega$ corresponding to different scale. Given $z\in[0,H_{\epsilon})^{d}$, we define a rectangular partition
$\mathcal{T}_{H_{\epsilon}}(z)$ with mesh size $H_{\epsilon}$ and
\[
\mathcal{T}_{H_{\epsilon}}(z):=\{K\subset\Omega|\;K=(x_{i}+(-\cfrac{H_{\epsilon}}{2},\cfrac{H_{\epsilon}}{2})^{d})\cap\Omega,\;\cfrac{x_{i}-z}{H_{\epsilon}}\in\mathbb{Z}^{d}\}.
\]
 For each coarse element $K_{z}^{(i)}=K_{z,H_{\epsilon}}(x_{i}):=(x_{i}+(-\cfrac{H_{\epsilon}}{2},\cfrac{H_{\epsilon}}{2})^{d})\cap\Omega\in\mathcal{T}_{H_{\epsilon}}(z)$,
we can divide the element into two subregions $K_{z,0}^{(i)}=K_{z}^{(i)}\cap\Omega_{0}$
and $K_{z,1}^{(i)}=K_{z}^{(i)}\cap\Omega_{1}$. To ensure the regularity
of the mesh, we combine the element $K_{z,H_{\epsilon}}(x_{i})$ with
$x_{i}\notin\Omega$ with their nearest element such that $\cup_{x_{i}\in(z+H_{\epsilon}\mathbb{Z}^{d})\cap\Omega}\bar{K}_{z,H_{\epsilon}}(x_{i})=\Omega$,
$K_{z,H_{\epsilon}}(x_{i})\cap K_{z,H_{\epsilon}}(x_{j})=\emptyset$
and aspect ratio of $K_{z,H_{\epsilon}}(x_{i})$ is smaller $\cfrac{\ensuremath{1}}{4}$
for all $x_{i}$. We then define the auxiliary basis $\psi_{i,0}^{z}=\psi_{0}^{z,H_{\epsilon}}(x_{i},x):=\chi_{K_{z,0}^{(i)}}$ and $\psi_{i,1}^{z}=\psi_{1}^{z,H_{\epsilon}}(x_{i},x):=\chi_{K_{z,1}^{(i)}}$
where $\chi_{Y}$ is the indicator function of $Y$ with $\chi_{Y}(x)=1\;\forall x\in Y$
and $\chi_{Y}(x)=0$ otherwise. For any $z\in[0,H_{\epsilon})^{d}$,
we define a index set $I_{z}=\{x_{i}|\;K_{z,H_{\epsilon}}(x_{i})\in\mathcal{T}_{H_{\epsilon}}(z)\}$ which contain all of the center points of the course elements in the partition $\mathcal{T}_{H_{\epsilon}}(z)$.
The partition $\mathcal{T}_{H}(z)$ with mesh size $H$ is defined in a similar way. The index $I_{z,H}$ is then defined correspondingly.

Following the idea of NLMC, we will define global basis functions $\phi_{i,0}^{z}\in H_{0}^{1}(\Omega)$
and $\phi_{i,1}^{z}\in H_{0}^{1}(\Omega)$ by corresponding constrained energy minimization problems in the $\mathcal{T}_{H_{\epsilon}}(z)$ such that 
\begin{align*}
a_{\epsilon}(\phi_{l,i}^{z},v)+\sum_{x_{k}\in I_{z}}\sum_{j=0,1}c_{k,j}(\psi_{k,j}^{z},v) & =0\\
(\phi_{l,i}^{z},\psi_{k,j}^{z}) & =\delta_{ij}\delta_{lk}\int_{\Omega}\psi_{k,j}^{z}.
\end{align*}
We denote the space $V_{glo,H_{\epsilon}}^{z}$ as the span of all global basis functions
$\phi_{k,i}^{z}$, that is, \[V_{glo,H_{\epsilon}}^{z}=\text{span}_{k,j}\{\phi_{k,j}^{z}\}=
\{v\in H_0^1(\Omega)|\; v=\sum_{k,j} v_{k,j}\phi_{k,j}^{z}, v_{k,j}\in \mathbb{R}.\}\] 

The following theorem is one of the main theorems in \cite{chung2018constraint},
 which gives
an $O(H_\epsilon)$ error bound for the CEM-GMsFEM method with global basis. 
\begin{theorem} \label{thm:NLMC}
Let $f\in L^{2}(\Omega)$ and $u_{glo,H_{\epsilon}}^{z}\in V_{glo,H_{\epsilon}}^{z}$ be the solution
of 
\[
a_{\epsilon}(u_{glo,H_{\epsilon}}^{z},v)=(f,v)\;\forall v\in V_{glo,H_{\epsilon}}^{z}.
\]
We have 
\[
\int_{\Omega}(u_{glo,H_{\epsilon}}^{z}-u_{\epsilon})\psi^z_{k,j}=0\;\forall k,j,
\]
and
\[
\|u_{glo,H_{\epsilon}}^{z}-u_{\epsilon}\|_{a}=O(H_{\epsilon}).
\]
\end{theorem}

As a corollary, we can write $u_{glo,H_{\epsilon}}^{z}$ in the form of $u_{glo,H_{\epsilon}}^{z}=\sum_{x_{l}\in I_{z}}\sum_{l}U_{l,i}^{H_{\epsilon}}\phi_{l,i}^{z}$
where the coefficients $U_{l,i}=\cfrac{(u_{\epsilon},\psi^z_{l,i})}{\int_{\Omega}\psi^z_{l,i}}$ are the local average of the solution $u_\epsilon$ in coarse elements $K_{z,i}^{(l)}$. 

\section{Multicontinumm homogenization} \label{sec:homo}
In this section, we will introduce the multicontinumm homogenization method and develop the proposed homogenized equation for the problem (\ref{eq:pde}).
This section is basically separated into two parts. The first part is the discussion of the downscaling operators using different multiscale features capturing basis functions such that the NLMC global basis functions. The next part introduces the multicontinumm homogenized equations using the downscaling operators and the original multiscale bilinear form with averaging.
To start our discussion, we will first define two downscaling operators $$P_{H_{\epsilon},H}^{z}:C^{1}(\Omega)^{2}\rightarrow H^{1}(\mathcal{T}_{H}(z)):=\{u\in L^{2}(\Omega)|u|_{K}\in H^{1}(K),\;\forall K\in\mathcal{T}_{H}(z)\}$$
$$P_{H_{\epsilon}}^{z}:C^{1}(\Omega)^{2}\rightarrow H^{1}(\Omega),$$
where
\[
P_{H_{\epsilon},H}^{z}(U_{0},U_{1}):=\sum_{i=0,1}\sum_{x_{l}\in I_{z,H}}\chi_{K_{z,H}(x_{l})}\Big(U_{i}(x_{l})\eta_{z}(x_{l},\cdot)+\sum_{k}\partial_{k}U_{i}(x_{l})\eta_{z,i}^{(k)}(x_{l},\cdot)\Big),
\]
\[
P_{H_{\epsilon}}^{z}(U_0,U_1):=\sum_{i=0,1}\sum_{x_{l}\in I_{z}}U_i(x_l)\phi_{l,i}^{z}.
\]
where $\eta_{i}$ and $\eta_{i}^{(l)}$ are defined as follows.

$\eta_{i}$ and $\eta_{i}^{(l)}$ are the constant-representing and linear-representing basis functions for each continuum which satisfy the following system
of equations 
\begin{align*}
\int_{K_{x,H}(x)^{+}}\kappa_{\epsilon}(y)\nabla_{y}\eta_{z,i}(x,y)\nabla_{y}v & =\sum_{x_{j}\in I_{x}\cap K_{x,H}(x)^{+}}\sum_{k=0,1}d_{j,k}(\psi_{k}(x_{j},y),v)\\
(\eta_{z,i}(x,y),\psi_{k}(x_{j},y)) & =\delta_{ki}\int_{K_{x,H}(x)^{+}}\psi_{k}(x_{j},y),\;\forall x_{j}\in I_{x}\cap K_{x,H}(x)^{+}
\end{align*}
and 
\begin{align*}
\int_{K_{x,H}(x)^{+}}\kappa_{\epsilon}(y)\nabla_{y}\eta_{z,i}^{(l)}(x,y)\nabla_{y}v & =\sum_{x_{j}\in I_{x}\cap K_{x,H}(x)^{+}}\sum_{k=0,1}d_{j,k}(\psi_{k}(x_{j},y),v)\\
(\eta_{z,i}^{(l)}(x,y),\psi_{k}(x_{j},y)) & =\delta_{ki}\int_{K_{x,H}(x)^{+}}(x_j^{(l)}-x^{(l)})\psi_{k}(x_{j},y),\;\forall x_{j}\in I_{x}\cap K_{x,H}(x)^{+},
\end{align*}
where $K_{x_l,H}(x_l)^+$ is the oversampling domain which enrich $K_{x_l,H}(x_l)$ by $k$ $H_\epsilon$-layer, $x^{(l)}$ is the $l$-th coordinate of $x\in\mathbb{R}^{d}$. 
We remark that the definition of the downscaling operators $P^z_{H_{\epsilon},H}$ and $P_{H_{\epsilon}}^{z}$ can be extended in a weaker sense by replacing $U(x_i)$ by local average of $U_i$, $\overline{U}_i(x_l):=\frac{\int_{K_{x_l,H_{\epsilon}}(x_l)}U_i}{|K_{x_l,H_{\epsilon}}(x_l)|}$. 

We next define the $L_2$ projection operator onto the auxiliary space $V^{z}_{aux,i}:=\text{span}_{l}\{\psi^z_{l,i}\}$ such that $\Pi_{0,H_{\epsilon}}^{z}$ and $\Pi_{1,H_{\epsilon}}^{z}$ satisfying 
\[
\Pi_{i,H_{\epsilon}}^{z}(v):=\sum_{x_{l}\in I_{z,H_{\epsilon}}}\cfrac{\int_{K_{z,H_{\epsilon}}(x_{l})\cap\Omega_{i}}v}{\int_{K_{z,H_{\epsilon}}(x_{l})\cap\Omega_{i}}1}\psi_{l,i}^{z}.
\]
Using theorem \ref{thm:NLMC}, we have $P_{H_{\epsilon}}^{z}(\Pi_{0,H_{\epsilon}}^{z}(u_{\epsilon}),\Pi_{1,H_{\epsilon}}^{z}(u_{\epsilon}))=u_{glo,H_{\epsilon}}^{z}$.

We will then use the downscaling operators to define the effective operator. Before introducing the effective operator, we will relate the local average of the summation, with respect to the reference point $z$, to the standard measure of domain $\Omega$. 
We consider $\rho_{\tau}$ be a smooth function such that
\[
\lim_{\tau}\rho_{\tau}(x-y)=\delta_{y}\;\text{in }C(\Omega)^{*}.
\]

We next define $\tilde{\rho}_{\tau}(x,z)=\sum_{x_{l}\in I_{z}}\rho_{\tau}(x-x_{l})$,
$\omega_{\tau}:\Omega\rightarrow\mathbb{R}$ such that
\[
\cfrac{1}{\omega_{\tau}(x)}=\int_{[0,H]^{d}}\tilde{\rho}_{\tau}(x,z)dz>0.
\]
and we can show $\omega=\lim_{\tau\rightarrow0}\omega_{\tau}=1$.
For any $g\in C(\Omega)$, we have 
\begin{equation}\label{eq:sum2int}
\begin{split}
\int_{[0,H]^{d}}\sum_{x_{l}\in I_{z,H}}g(x_{l})dz & =\lim_{\tau}\int_{[0,H]^{d}}\int_{\Omega}\tilde{\rho}_{\tau}(x,z)g(x)dxdz\\
 & =\lim_{\tau}\int_{\Omega}\cfrac{1}{\omega_{\tau}(x)}g(x)dx\\
 & = \int_{\Omega}g(x)dx. 
 \end{split}
\end{equation}
Thus, we can extend the definition of $\int_{z\in[0,H]^{d}}\sum_{x_{l}\in I_{z,H}}$ in a weaker sense.

We next define a bilinear operator $\tilde{a}^z_{glo,H_\epsilon}:C^{1}(\Omega)^{2}\times C^{1}(\Omega)^{2}\rightarrow\mathbb{R}$ by
\begin{align*}
\tilde{a}^z_{glo,H_\epsilon}((U_0,U_1),(V_0,V_1))&:= a_{\epsilon}(P_{H_{\epsilon}}^{z}(U_0,U_1),P_{H_{\epsilon}}^{z}(V_0,V_1)).
\end{align*}
These bilinear operators can be extended in a weaker sense such that the bilinear operators are well-defined in $H^1(\Omega)^2$. 
By theorem \ref{thm:NLMC}, we can obtain 
\begin{equation} \label{eq:NLMC_residual}
    \tilde{a}^z_{glo,H_\epsilon}((\Pi_{0,H_{\epsilon}}^{z}(u_{\epsilon}),\Pi_{1,H_{\epsilon}}^{z}(u_{\epsilon})),(\Pi_{0,H_{\epsilon}}^{z}(v),\Pi_{1,H_{\epsilon}}^{z}(v))) = (f,P_{H_{\epsilon}}^{z}(\Pi_{0,H_{\epsilon}}^{z}(v),\Pi_{1,H_{\epsilon}}^{z}(v))),
\end{equation}
for any $v\in H^1_0(\Omega)$, and $\|P_{H_{\epsilon}}^{z}(\Pi_{0,H_{\epsilon}}^{z}(u_{\epsilon}),\Pi_{1,H_{\epsilon}}^{z}(u_{\epsilon})) - u_\epsilon\|_a\leq CH_\epsilon$.
Therefore, if we define an effective bilinear operator as $\tilde{a}^z_{glo, H_\epsilon}(\cdot,\cdot)$, the projection of the multiscale solution $u_\epsilon$ satisfies the upscaled equation in the sense of (\ref{eq:NLMC_residual}).
However, this effective operator is not written as a form of system of differential equations and it brings difficulties to the analysis of the effective model and the numerical implementation. In additional, these effective operators are not strongly coercive with a given $z$.

We will next define a homogenized bilinear operator given in a form of partial differential operator using a linear approximation basis functions $\eta_{x,i},\eta^{(k)}_{x,i}$ and a local averaging in computational scale, $H$. The homogenized bilinear operator $\tilde{a}_{H_{\epsilon},H}:H^{1}(\Omega)^{2}\times H^{1}(\Omega)^{2}\rightarrow\mathbb{R}$ is defined as
\begin{equation}
\begin{split}
\tilde{a}_{H_{\epsilon},H}((U_{0},U_{1}),(V_{0},V_{1}))  :=\\
\int_{\Omega}\cfrac{1}{H^{d}}\int_{K_{x,H}(x)}\kappa(y)\nabla_{y}\sum_{i}\Big(U_{i}(x)\eta_{x,i}(x,y)+\sum_{k}\partial_{k}U_{i}(x)\eta_{x,i}^{(k)}(x,y)\Big)\cdot\\
\nabla_{y}\sum_{j}\Big(V_{j}(x)\eta_{x,j}(x,y)+\sum_{k}\partial_{k}V_{j}(x)\eta_{x,j}^{(k)}(x,y)\Big)dydx\\
  = \int_{\Omega}\Big(\sum_{i,j,k,l}\alpha_{kl,H_{\epsilon},H}^{ij}(x)\partial_{k}U_{i}(x)\partial_{l}V_{j}(x)+\sum_{i,j,k}\beta_{k,H_{\epsilon},H}^{ij}(x)\partial_{k}U_{i}(x)V_{j}(x)+\\
\sum_{i,j,k}\beta_{k,H_{\epsilon},H}^{ji}(x)\partial_{k}V_{j}(x)U_{i}(x)+\sum_{i,j}\gamma_{H_{\epsilon},H}^{ij}(x)U_{i}(x)V_{j}(x)\Big)dx,
\end{split} \label{eq:upscale_operator}
\end{equation}
where 
\begin{align*}
\alpha_{kl,H_{\epsilon},H}^{ij}(x) & =\cfrac{1}{H^{d}}\int_{K_{x,H}(x)}\kappa(y)\nabla_{y}\eta_{x,i}^{(k)}(x,y)\cdot\nabla_{y}\eta_{x,j}^{(l)}(x,y)dy\\
\beta_{k,H_{\epsilon},H}^{ij}(x) & =\cfrac{1}{H^{d}}\int_{K_{x,H}(x)}\kappa(y)\nabla_{y}\eta_{x,i}^{(k)}(x,y)\cdot\nabla_{y}\eta_{x,j}(x,y)dy\\
\gamma_{H_{\epsilon},H}^{ij}(x) & =\cfrac{1}{H^{d}}\int_{K_{x,H}(x)}\kappa(y)\nabla_{y}\eta_{x,i}(x,y)\cdot\nabla_{y}\eta_{x,j}(x,y)dy.
\end{align*}

We can check that the effective operator $\tilde{a}_{H_{\epsilon},H}$ is an averaging of the multiscale operator $a_{\epsilon}$ in the sense of 
\begin{equation}
  \tilde{a}_{H_{\epsilon},H}((U_{0},U_{1}),(V_{0},V_{1}))=\cfrac{1}{H^{d}}\int_{[0,H]^{d}}a_{\epsilon}(P_{H_{\epsilon},H}^{z}(U_{0},U_{1}),P_{H_{\epsilon},H}^{z}(V_{0},V_{1}))dz.
  \label{eq:aver_eq}
\end{equation}
We draft the proof of Equation (\ref{eq:aver_eq}) as below. 
Since
\begin{equation}
\begin{split}
  a_{\epsilon}\Big(\sum_{i=0,1}\sum_{x_{l}\in I_{z,H}}\chi_{K_{z,H}(x_{l})}\Big(U_{i}(x_{l})\eta_{z,i}(x_{l},\cdot)+\sum_{k}\partial_{k}U_{i}(x_{l})\eta_{z,i}^{(k)}(x_{l},\cdot)\Big),\\
\sum_{i=0,1}\sum_{x_{l}\in I_{z,H}}\chi_{K_{z,H}(x_{l})}\Big(V_{i}(x_{l})\eta_{z,i}(x_{l},\cdot)+\sum_{k}\partial_{k}V_{i}(x_{l})\eta_{z,i}^{(k)}(x_{l},\cdot)\Big)\Big)\\
=  \sum_{i=0,1}\sum_{j=0,1}\sum_{x_{l}\in I_{z,H}}a_{\epsilon}(\chi_{K_{z,H}(x_{l})}(U_{i}(x_{l})\eta_{z,i}(x_{l},\cdot)+\sum_{k}\partial_{k}U_{i}(x_{l})\eta_{z,i}^{(k)}(x_{l},\cdot)),\\
\chi_{K_{z,H}(x_{l})}(V_{j}(x_{l})\eta_{z,j}(x_{l},\cdot)+\sum_{k}\partial_{k}V_{j}(x_{l})\eta_{z,j}^{(k)}(x_{l},\cdot))).
\end{split}
\end{equation}
to show the equation (\ref{eq:aver_eq}), we only need to check 
\begin{equation} \label{eq:ind_aver_eq}
\begin{split}
&\int_{\Omega}\Big(\gamma_{H_{\epsilon},H}^{ji}U_{i}V_{j}\Big)dx \\
=&\cfrac{1}{H^{d}}\int_{[0,H]^{d}}\sum_{x_{l}\in I_{z,H}}a_{\epsilon}\Big(\chi_{K_{z,H}(x_{l})}U_{i}(x_{l})\eta_{z,i}(x_{l},\cdot),\chi_{K_{z,H}(x_{l})}V_{j}(x_{l})\eta_{z,j}(x_{l},\cdot)\Big)dz,\\
&\int_{\Omega}\Big(\beta_{k,H_{\epsilon},H}^{ji}\partial_{k}U_{i}V_{j}\Big)dx\\
=&\cfrac{1}{H^{d}}\int_{[0,H]^{d}}\sum_{x_{l}\in I_{z,H}}a_{\epsilon}\Big(\chi_{K_{z,H}(x_{l})}\partial_{k}U_{i}(x_{l})\eta_{z,i}^{(k)}(x_{l},\cdot),\chi_{K_{z,H}(x_{l})}V_{j}(x_{l})\eta_{z,j}(x_{l},\cdot)\Big)dz,\\
&\int_{\Omega}\Big(\alpha_{ks,H_{\epsilon},H}^{ij}\partial_{k}U_{i}\partial_{s}V_{j}\Big)dx\\
=&\cfrac{1}{H^{d}}\int_{[0,H]^{d}}\sum_{x_{l}\in I_{z,H}}a_{\epsilon}\Big(\chi_{K_{z,H}(x_{l})}\partial_{k}U_{i}(x_{l})\eta_{z,i}^{(k)}(x_{l},\cdot),\chi_{K_{z,H}(x_{l})}\partial_{s}V_{j}(x_{l})\eta_{z,j}^{(s)}(x_{l},\cdot)\Big)  dz.
\end{split}
\end{equation}

We will only show one of the equations. The proof of other equations follows a similar idea. Since $U_i(x_l),V_j(x_l)$ are independent of the integrating variable of $a_\epsilon(\cdot,\cdot)$, we obtain
\begin{equation}\label{eq:local_equation}
\begin{split}
&a_{\epsilon}\Big(\chi_{K_{z,H}(x_{l})}U_{i}(x_{l})\eta_{z}(x_{l},\cdot),\chi_{K_{z,H}(x_{l})}V_{j}(x_{l})\eta_{z,j}(x_{l},\cdot)\Big) \\
=&U_{i}(x_{l})V_{j}(x_{l})a_{\epsilon}\Big(\chi_{K_{z,H}(x_{l})}\eta_{z,i}(x_{l},\cdot),\chi_{K_{z,H}(x_{l})}\eta_{z,j}(x_{l},\cdot)\Big)\\
 =&U_{i}(x_{l})V_{j}(x_{l})\int_{K_{z,H}(x_{l})}\kappa(y)\nabla_{y}\eta_{z,i}(x_{l},y)\cdot\nabla_{y}\eta_{z,j}(x_{l},y)dy.
\end{split}
\end{equation}
Using Equations (\ref{eq:sum2int}) and (\ref{eq:local_equation}), we have
\begin{equation}
\begin{split}
&\cfrac{1}{H^{d}}\int_{[0,H]^{d}}\sum_{x_{l}\in I_{z,H}}a_{\epsilon}(\chi_{K_{z,H}(x_{l})}U_{i}(x_{l})\eta_{z,i}(x_{l},\cdot),\chi_{K_{z,H}(x_{l})}V_{j}(x_{l})\eta_{z,j}(x_{l},\cdot)\Big)dz\\
=&\cfrac{1}{H^{d}}\int_{[0,H]^{d}}\sum_{x_{l}\in I_{z,H}}U_{i}(x_{l})V_{i}(x_{l})\int_{K_{z,H}(x_{l})}\kappa(y)\nabla_{y}\eta_{z,i}(x,y)\cdot\nabla_{y}\eta_{z,j}(x,y)dydz\\
  =&\cfrac{1}{H^{d}}\int_{[0,H]^{d}}\sum_{x_{l}\in I_{z,H}}U_{i}(x_{l})V_{j}(x_{l})\gamma_{H_{\epsilon},H}^{ij}(x_{l})dz=\int_{\Omega}\gamma_{H_{\epsilon},H}^{ij}U_{i}V_{j}dx.
\end{split}
\end{equation}
Therefore, the first equation in (\ref{eq:ind_aver_eq}) is verified. 

After defining the effective operator, we can introduce the numerical homogenized solution $(U_{0,H_{\epsilon},H},U_{1,H_{\epsilon},H})$. We consider $(U_{0,H_{\epsilon},H},U_{1,H_{\epsilon},H})\in[H^{1}(\Omega)]^{2}$ to be the solution of the upscaled problem satisfying 
\begin{equation}
\tilde{a}_{H_{\epsilon},H}((U_{0,H_{\epsilon},H},U_{1,H_{\epsilon},H}),(V_{0},V_{1}))=\cfrac{1}{H^{d}}\int_{[0,H]^{d}}(f,P_{H_{\epsilon},H}^{z}(V_{0},V_{1}))\;\forall (V_{0},V_{1})\in[H^{1}(\Omega)]^{2} \label{eq:homo_eq}
\end{equation} 
\section{Analysis}\label{sec:analy}

In this section, we will present an analysis of the proposed upscaling method.
The proof is outlined as follows: We will first analyze the error estimate between the downscaling operators. The next step will be estimating the residual of the proposed homogenized equation. The very last step is applying this result and obtaining a convergence result of the homogenized solution.

Before, starting our analysis, for a given $\alpha>0$, we define a norm $\|\cdot\|_{a,1+\alpha}$ for the space $H^{1+\alpha}(\Omega)^2$ where
\[
\|(U_{0},U_{1})\|_{a,1+\alpha}^{2}=\|U_{0}\|_{H^{1+\alpha}}^{2}+\kappa_{max}\|U_{1}\|_{H^{1+\alpha}}^{2}.
\]

In the following lemma, we will show that the NLMC downscaling operator $P_{H_{\epsilon}}^{z}$ with a scale $H_{\epsilon} \ll H$ can be approximated by the multicontinuum homogenization downscaling operator $\tilde{P}_{H_{\epsilon}, H}^{z}$ if $(U_0, U_1)$ is smooth enough.
\begin{lemma}
\label{lem3}
Given $U_{0},U_{1}\in H^{1+\alpha}(\Omega)$ with $0<\alpha\leq 1$, the number of the oversampling layer If $2kH_{\epsilon}<H=O(kH_{\epsilon})$ and $C_0\geq k= O(\log(1/H_{\epsilon}))$ for some $C_0>0$, we have 
\[
\|P_{H_{\epsilon},H}^{z}(U_{0},U_{1})-P_{H_{\epsilon}}^{z}(U_{0},U_{1})\|_{a}\leq C\log(H^{-1})H^{\alpha}\|(U_{0},U_{1})\|_{a,1+\alpha}.
\]
Moreover, we have
\[
\|P_{H_{\epsilon},H}^{z}(U_{0},U_{1})-P_{H_{\epsilon}}^{z}(U_{0},U_{1})\|_{L^{2}}\leq C\log(H^{-1})H^{1+\alpha}\|(U_{0},U_{1})\|_{a,1+\alpha}.
\]
\end{lemma}

\begin{proof}
By the definition of $P_{H_{\epsilon},H}^{z}$, we have 
\begin{equation}
\begin{split}
&a_{\epsilon}(P_{H_{\epsilon}}^{z}(U_{0},U_{1})-P_{H_{\epsilon},H}^{z}(U_{0},U_{1}),v)\\
=&\sum_{x_{l}\in I_{z,H}}a_{\epsilon}\Big(\chi_{K_{z,H}(x_{l})}\Big(P_{H_{\epsilon}}^{z}(U_{0,H_{\epsilon}},U_{1,H_{\epsilon}})-\sum_{i=0,1}\Big(U_{i}(x_{l})\eta_{z}(x_{l},\cdot)+\sum_{k}\partial_{k}U_{i}(x_{l})\eta_{z,i}^{(k)}(x_{l},\cdot)\Big)\Big),v\Big)
\end{split}
\end{equation}
For each $K_{z,H}(x_l)$, we can define $P_{H,H_{\epsilon}}^{z,x_{l},loc}(U_{0},U_{1})\in H_{0}^{1}(K_{z,H}(x_l)^{+})$ such that
\begin{align*}
\int_{K_{z,H}(x_l)^{+}}\kappa \nabla P_{H,H_{\epsilon}}^{z,x_{l},loc}(U_{0},U_{1})\cdot \nabla v & =\sum_{x_{k}\in I_{z}\cap K_{z,H_\epsilon}(x_l)^{+}}\sum_{j=0,1}c_{k,j}(\psi_{k,j}^{z},v)\\
\int_{K_{z,H}(x_l)^{+}}P_{H,H_{\epsilon}}^{z,x_{l},loc}(U_{0},U_{1})\psi_{k,j}^{z} & =\delta_{ij}\int_{K_{z,H}(x_l)^{+}}U_i\psi_{k,j}^{z}\;\forall x_k\in I_{z}\cap K_{z,H}(x_l)^{+}.
\end{align*}
In \cite{chung2018constraint}, the authors show that $\phi_{l,i}^{z}$ have exponential decaying property. Using the exponential decaying property, we have
\begin{equation}
\begin{split}
&\|P_{H_{\epsilon}}^{z}(U_{0},U_{1})-P_{H,H_{\epsilon}}^{z,x_{l},loc}(U_{0},U_{1})\|_{a(K_{z,H}(x_{l}))}^{2}\\ \leq & C(1+E)^{-k}\|P_{H_{\epsilon}}^{z}(U_{0},U_{1})\|_{a(K_{z}^{k}(x_{l})^+)}^{2}.
\end{split}
\end{equation}
Thus, we can obtain
\begin{equation}
\begin{split}
|a(P_{H_{\epsilon}}^{z}(U_{0},U_{1})-P_{H_{\epsilon},H}^{z}(U_{0},U_{1}),v)|\leq & \sum_{x_{l}\in I_{z,H}}|a_{\epsilon}(\chi_{K_{z,H}(x_{l})}\Big(P_{H,H_{\epsilon}}^{z,x_{l},loc}(U_{0},U_{1})\\
&-\sum_{i=0,1}\Big(U_{i}(x_{l})\eta_{z}(x_{l},\cdot)+\sum_{k}\partial_{k}U_{i}(x_{l})\eta_{z,i}^{(k)}(x_{l},\cdot)\Big)\Big),v)|\\
 & +Ck^{\frac{d}{2}}(1+E)^{-\frac{k}{2}}\|P_{H_{\epsilon}}^{z}(U_{0},U_{1})\|_{a}\|v\|_{a}.
\end{split}
\end{equation}

We next define the local error function $e_{l,z}$ as
\begin{equation*}
\begin{split}
e_{l,z} &:=P_{H,H_{\epsilon}}^{z,x_{l},loc}(U_{0},U_{1}) -\sum_{i=0,1}\Big(U_{i}(x_{l})\eta_{z}(x_{l},\cdot)+\sum_{k}\partial_{k}U_{i}(x_{l})\eta_{z,i}^{(k)}(x_{l},\cdot) \Big)\in V_{0}(K_{z,H}^{+}(x_{l})).
\end{split}
\end{equation*}
To estimate the bound of the error function, using the property that both of the terms, $\;$
$P_{H,H_{\epsilon}}^{z,x_{l},loc}(U_{0},U_{1})$ and $\sum_{i=0,1}\Big(U_{i}(x_{l})\eta_{z}(x_{l},\cdot)+\sum_{k}\partial_{k}U_{i}(x_{l})\eta_{z,i}^{(k)}(x_{l},\cdot) \Big)$ satisfying the energy minimization problem with specified constraints, the error function satisfies the following system of equations:
\begin{equation*}
\begin{split}
\int_{K_{z,H}^{+}(x)}\kappa_{\epsilon}(y)\nabla_{y}e_{l,z}\nabla_{y}v  &=\sum_{x_{j}\in I_{x,H_{\epsilon}}\cap K_{z,H}^{+}(x_{l})}\sum_{k=0,1}\mu_{j,k}^{e}(\psi_{k}(x_{j},y),v)\\
\int_{K_{z,H}^{+}(x)} e_{l,z}(y)\psi_{m}(x_{j},y)dy  &=
\int_{K_{z,H}^{+}(x_{l})}\Big(U_{i}(x_j)-U_{i}(x_{l})-\sum_{k}(x_j^{(k)}-x_l^{(k)})\partial_{k}U_{i}(x_{l})\Big)\psi_{m}(x_{j},y)dy,\\
&\hspace{6cm}\forall x_{j}\in I_{x,H_{\epsilon}}\cap K_{z,H}^{+}(x_{l}).
\end{split}
\end{equation*}
Using the regularity of $U_1,U_2$, we obtain that
\[
\|U_{i}(x_j)-U_{i}(x_{l})-\sum_{k}(x_j^{(k)}-x_l^{(k)})\partial_{k}U_{i}(x_{l})\|_{L^{2}(K_{z,H}^{+}(x_{l}))}\leq C(H+kH_{\epsilon})^{1+\alpha}\|U_{i}\|_{H^{1+\alpha}(K_{z,H}^{+}(x_{l}))}.
\]
By the stability of the local problem shown in \cite{chung2018constraint}, we have 
\[
\|e_{l,z}\|_{a(K_{z,H}^{+}(x_{l}))}^{2}\leq C(\cfrac{H}{H_{\epsilon}}+k)^{2}(H+kH_{\epsilon})^{2\alpha}\Big(\|U_{0}\|_{H^{1+\alpha}(K_{z,H}^{+}(x_{l}))}^{2}+\kappa_{\max}\|U_{1}\|_{H^{1+\alpha}(K_{z,H}^{+}(x_{l}))}^{2}\Big)
\]
and 
\[\|e_{l,z}\|_{a}^{2}
\leq C(\cfrac{H}{H_{\epsilon}}+k)^{2}(H+kH_{\epsilon})^{2\alpha}\sum_{x\in I_{z,H}}\Big(\|U_{0}\|_{H^{1+\alpha}(K_{z,H}^{+}(x_{l}))}^{2}+\kappa_{\max}\|U_{1}\|_{H^{1+\alpha}(K_{z,H}^{+}(x_{l}))}^{2}\Big)
\]
If $kH_{\epsilon}<H/2$, we have $D_{H_{\epsilon},H}=\max_{y,z}\{\sum_{x_{l}\in I_{z,H}}\chi_{K_{z,H}^{+}(x_{l})}(y)\}\leq D<\infty$
and 
\[
\sum_{x\in I_{z,H}}\|U_{i}\|_{H^{1+\alpha}(K_{z,H}^{+}(x_{l}))}^{2}\leq D\sum_{x\in I_{z,H}}\|U_{i}\|_{H^{1+\alpha}(K_{z,H}(x_{l}))}^{2}=D\|U_{i}\|_{H^{1+\alpha}(\Omega)}^{2}
\]
Therefore, if $H>2kH_{\epsilon}$ and $k=C_0\log(1/H_{\epsilon})$ with
large enough $C_{0}>0$, we have 
\[
\|e_{l,z}\|_{a}^{2}\leq CD\log(1/H_{\epsilon})^{2}H^{2\alpha}\Big(\|U_{0}\|_{H^{1+\alpha}(\Omega)}^{2}+\kappa_{\max}\|U_{1}\|_{H^{1+\alpha}(\Omega)}^{2}\Big).
\]

We will then show the $L^2$ estimate of error function $e_{l,z}$. Using the Poincare's inequality for the projection operator, we have
\[
\|e_{l,z}-\Pi^z e_{l,z}\|_{L^2(K_{z,H_{\epsilon}}(x))}^{2}\leq CH_{\epsilon}^{2}\|e_{l,z}-\Pi^z e_{l,z}\|_{a(K_{z,H_{\epsilon}}(x))}^{2}
\]
for all $x\in I_{z,H_{\epsilon}}\cap K_{z,H}^{+}(x_{l})$ where $\Pi^z= \sum_{i=1,2}\Pi_{i,H_\epsilon}^z$. 
Using the constraints of the local problem, we have 
\[
\Pi^z (e_{l,z}) = \sum_i\Big(U_{i}(x_j)-U_{i}(x_{l})-\sum_{k}(x_j^{(k)}-x_{l}^{(k)})\partial_{k}U_{i}(x_{l})\Big)
\] and
\[
\|\Pi^z (e_{l,z})\|^2_{L^2(K_{z,H_{\epsilon}}(x))} \leq \|\Big(U_{i}(y)-U_{i}(x_{l})-\sum_{k}(y^{(k)}-x_{l}^{(k)})\partial_{k}U_{i}(x_{l})\Big)\|^2_{L^2(K_{z,H_{\epsilon}}(x))}.
\]
Therefore, the $L^2$ estiamte is given by
\begin{align*}
\|e_{l,z}\|_{L^2(K_{z,H}(x_{l}))}^{2} & \leq C(H+kH_{\epsilon})^{2+2\alpha}\Big(\|U_{0}\|_{H^{1+\alpha}(\Omega)}^{2}+\kappa_{\max}\|U_{1}\|_{H^{1+\alpha}(\Omega)}^{2}\Big)+H_{\epsilon}^{2}\|e_{l,z}\|_{a}^{2}\\
 & \leq CH{}^{2+2\alpha}\Big(\|U_{0}\|_{H^{1+\alpha}(\Omega)}^{2}+\kappa_{\max}\|U_{1}\|_{H^{1+\alpha}(\Omega)}^{2}\Big).
\end{align*}
\end{proof}

We consider an operator $A_{H_{\epsilon},H}^{-1,*}:L^{2}\rightarrow[H^{1}]^{2}$
as
\[
\tilde{a}_{H_{\epsilon},H}(A_{H_{\epsilon},H,0}^{-1,*}(f),(V_{0},V_{1}))=\cfrac{1}{H^{d}}\int_{[0,H]^{d}}(f,P_{H_{\epsilon},H}^{z}(V_{0},V_{1}))\;\forall(V_{0},V_{1})\in[H^{1}]^{2}.
\]
In this paper, we assume the smoothness of the upscaling quantities. We assume for all $\epsilon>0$  there exist a smooth function $u_{0,\epsilon},u_{1,\epsilon}\in H^{1+\alpha}(\Omega)$ such that $\Pi_{0,H_{\epsilon}}^{z}(u_{\epsilon})=\Pi_{0,H_{\epsilon}}^{z}(u_{0,\epsilon})$ and $\Pi_{1,H_{\epsilon}}^{z}(u_{\epsilon})=\Pi_{1,H_{\epsilon}}^{z}(u_{1,\epsilon})$ with $\|u_{i,\epsilon}\|_{H^{1+\alpha}}<C$ with an uniform constant $C$.
In order to obtain the numerical homogenized solution satisfying the smoothness assumption, we assume the inverse of the effective operator has the following smoothing property.
\begin{assumption} \label{ass:regularity} We assume $A_{H_{\epsilon},H}^{-1,*}$ is a bounded operator mapping
from $L^{2}$ to $[H^{1+\alpha}]^{2}$ such that 
\[
\|A_{H_{\epsilon},H}^{-1,*}(f)\|_{a,1+\alpha}\leq C_\zeta\|f\|_{L^{2}}.
\]
\end{assumption}
In the following lemma, we will give an estimate of the residual of the exact solution for the homogenized equations (\ref{eq:homo_eq}) in the following sense.
\begin{lemma}
Let $f\in L^{2}(\Omega)$ and $u_{\epsilon}$ be the solution of the
PDE. If If $2kH_{\epsilon}<H=O(kH_{\epsilon})$ and $C_0\geq k= O(\log(1/H_{\epsilon}))$ for some $C_0>0$, we have
\[
\tilde{a}_{H_{\epsilon},H}((U_{0,H_{\epsilon},H}-\Pi_{0,H_{\epsilon}}^{z}(u_{\epsilon}),U_{1,H_{\epsilon},H}-\Pi_{1,H_{\epsilon}}^{z}(u_{\epsilon})),(V_{0},V_{1}))\leq C\log(1/H_{\epsilon})H^{\alpha}\|f\|_{L^{2}}\|(V_{0},V_{1})\|_{a,1+\alpha}.
\]
\end{lemma}

\begin{proof}
We have 
\begin{align*}
 & \cfrac{1}{H^{d}}\int_{[0,H]^{d}}a_{\epsilon}(P_{H_{\epsilon},H}^{z}(U_{0,H_{\epsilon},H}-\Pi_{0,H_{\epsilon}}^{z}(u_{\epsilon}),U_{1,H_{\epsilon},H}-\Pi_{1,H_{\epsilon}}^{z}(u_{\epsilon})),P_{H_{\epsilon},H}^{z}(V_{0},V_{1}))\\
= & \cfrac{1}{H^{d}}\int_{[0,H]^{d}}a_{\epsilon}(P_{H_{\epsilon},H}^{z}(U_{0,H_{\epsilon},H},U_{1,H_{\epsilon},H})-P_{H_{\epsilon}}^{z}(\Pi_{0,H_{\epsilon}}^{z}(u_{\epsilon}),\Pi_{1,H_{\epsilon}}^{z}(u_{\epsilon})),P_{H_{\epsilon},H}^{z}(V_{0},V_{1}))\\
 & +\cfrac{1}{H^{d}}\int_{[0,H]^{d}}a_{\epsilon}(P_{H_{\epsilon}}^{z}(\Pi_{0,H_{\epsilon}}^{z}(u_{\epsilon}),\Pi_{1,H_{\epsilon}}^{z}(u_{\epsilon}))-P_{H_{\epsilon},H}^{z}(\Pi_{0,H_{\epsilon}}^{z}(u_{\epsilon}),\Pi_{1,H_{\epsilon}}^{z}(u_{\epsilon})),P_{H_{\epsilon},H}^{z}(V_{0},V_{1}))
\end{align*}
By the definitions of $P_{H_{\epsilon},H}^{z}(U_{0,H_{\epsilon},H},U_{1,H_{\epsilon},H}),P_{H_{\epsilon}}^{z}(\Pi_{0,H_{\epsilon}}^{z}(u_{\epsilon}),\Pi_{1,H_{\epsilon}}^{z}(u_{\epsilon}))$
and Lemma \ref{lem3}, we have
\begin{equation}
\begin{split}
  &\cfrac{1}{H^{d}}\int_{[0,H]^{d}}a_{\epsilon}(P_{H_{\epsilon},H}^{z}(U_{0,H_{\epsilon},H},U_{1,H_{\epsilon},H})-P_{H_{\epsilon}}^{z}(\Pi_{0,H_{\epsilon}}^{z}(u_{\epsilon}),\Pi_{1,H_{\epsilon}}^{z}(u_{\epsilon})),P_{H_{\epsilon},H}^{z}(V_{0},V_{1}))\\
=&  \cfrac{1}{H^{d}}\int_{[0,H]^{d}}(f,P_{H_{\epsilon},H}^{z}(V_{0},V_{1})-P_{H_{\epsilon}}^{z}(V_{0},V_{1}))\\
&-\int_{[0,H]^{d}}a_{\epsilon}(P_{H_{\epsilon}}^{z}(\Pi_{0,H_{\epsilon}}^{z}(u_{\epsilon}),\Pi_{1,H_{\epsilon}}^{z}(u_{\epsilon}),P_{H_{\epsilon},H}^{z}(V_{0},V_{1})-P_{H_{\epsilon}}^{z}(V_{0},V_{1}))\\
\leq & CH^{1+\alpha}\|f\|_{L^{2}}\|(V_{0},V_{1})\|_{a,1+\alpha}+C\log(1/H_{\epsilon})H^{\alpha}\|P_{H_{\epsilon}}^{z}(\Pi_{0,H_{\epsilon}}^{z}(u_{\epsilon}),\Pi_{1,H_{\epsilon}}^{z}(u_{\epsilon})\|_{a}\|(V_{0},V_{1})\|_{a,1+\alpha}
\end{split}
\end{equation}
and
\begin{align*}
 & \cfrac{1}{H^{d}}\int_{[0,H]^{d}}a_{\epsilon}(\tilde{P}_{H_{\epsilon}}^{z}(\Pi_{0,H_{\epsilon}}^{z}(u_{\epsilon}),\Pi_{1,H_{\epsilon}}^{z}(u_{\epsilon}))-P_{H_{\epsilon},H}^{z}(\Pi_{0,H_{\epsilon}}^{z}(u_{\epsilon}),\Pi_{1,H_{\epsilon}}^{z}(u_{\epsilon})),P_{H_{\epsilon},H}^{z}(V_{0},V_{1}))\\
\leq & CH^{\alpha}\|(u_{0,\epsilon},u_{1,\epsilon})\|_{a,1+\alpha}\|(V_{0},V_{1})\|_{a}.
\end{align*}
Therefore, we have 
\[
\tilde{a}_{H_{\epsilon},H}((U_{0,H_{\epsilon},H}-\Pi_{0,H_{\epsilon}}^{z}(u_{\epsilon}),U_{1,H_{\epsilon},H}-\Pi_{1,H_{\epsilon}}^{z}(u_{\epsilon})),(V_{0},V_{1}))\leq CH^{\alpha}\|f\|_{L^{2}}\|(V_{0},V_{1})\|_{a,1+\alpha}.
\]

\end{proof}

Note that from this theorem, it follows that if $U_i$'s are smooth
(one of main assumptions in multicontinuum homogenization), we have
\[
\tilde{a}_{H_{\epsilon},H}((U_{0,H_{\epsilon},H}-\Pi_{0,H_{\epsilon}}^{z}(u_{\epsilon}),U_{1,H_{\epsilon},H}-\Pi_{1,H_{\epsilon}}^{z}(u_{\epsilon})),(U_{0},U_{1}))\leq C\log(\cfrac{1}{H_{\epsilon}})H^{\alpha}\|f\|_{L^{2}}\|(U_{0},U_{1})\|_{a,1+\alpha}
\]
and thus small. Next, we show that if one assumes that the inverse operator
is bounded, then we can get an estimate for the solution.

\begin{theorem}
Let $f\in L^{2}(\Omega)$ and $u_{\epsilon}$ be the solution of the
PDE. If $2kH_{\epsilon}<H=O(kH_{\epsilon})$ and $C_0\geq k= O(\log(1/H_{\epsilon}))$ for some $C_0>0$, we have 
\[
\|u_{\epsilon}-\cfrac{1}{H^{d}}\int_{[0,H]^{d}}P_{H_{\epsilon},H}^{z}(U_{0,H_{\epsilon}},U_{1,H_{\epsilon}})dz\|_{L^{2}}\leq C\log(1/H_{\epsilon})H^{\alpha}\|f\|_{L^{2}}.
\]
\end{theorem}

\begin{proof}

We consider $(V_{0},V_{1})=A_{H_{\epsilon},H}^{-1,*}(\cfrac{1}{H^{d}}\int_{[0,H]^{d}}P_{H_{\epsilon},H}^{z}(U_{0,H_{\epsilon},H}-\Pi_{0,H_{\epsilon}}^{z}(u_{\epsilon}),U_{1,H_{\epsilon},H}-\Pi_{1,H_{\epsilon}}^{z}(u_{\epsilon})))$
and obtain 
\begin{equation}
\begin{split}
&\|\cfrac{1}{H^{d}}\int_{[0,H]^{d}}P_{H_{\epsilon},H}^{z}(U_{0,H_{\epsilon},H}-\Pi_{0,H_{\epsilon}}^{z}(u_{\epsilon}),U_{1,H_{\epsilon},H}-\Pi_{1,H_{\epsilon}}^{z}(u_{\epsilon}))\|_{L^{2}}^{2}  \\
=&\tilde{a}_{H_{\epsilon},H}((U_{0,H_{\epsilon},H}-\Pi_{0,H_{\epsilon}}^{z}(u_{\epsilon}),U_{1,H_{\epsilon},H}-\Pi_{1,H_{\epsilon}}^{z}(u_{\epsilon})),(V_{0},V_{1}))\\
  \leq & C\log(1/H_{\epsilon})H^{\alpha}\|f\|_{L^{2}}\|(V_{0},V_{1})\|_{a,1+\alpha}\\
  \leq &C\log(1/H_{\epsilon})H^{\alpha}\|f\|_{L^{2}}\|\cfrac{1}{H^{d}}\int_{[0,H]^{d}}P_{H_{\epsilon},H}^{z}(U_{0,H_{\epsilon},H}-\Pi_{0,H_{\epsilon}}^{z}(u_{\epsilon}),U_{1,H_{\epsilon},H}-\Pi_{1,H_{\epsilon}}^{z}(u_{\epsilon}))\|_{L^{2}}.
\end{split}
\end{equation}

\begin{equation}
\begin{split}
&\|u_{\epsilon}-\cfrac{1}{H^{d}}\int_{[0,H]^{d}}P_{H_{\epsilon},H}^{z}(U_{0,H_{\epsilon},H},U_{1,H_{\epsilon},H})\|_{L^{2}} \\
\leq&\|u_{\epsilon}-\cfrac{1}{H^{d}}\int_{[0,H]^{d}}P_{H_{\epsilon}}^{z}(\Pi_{0,H_{\epsilon}}^{z}(u_{\epsilon}),\Pi_{1,H_{\epsilon}}^{z}(u_{\epsilon}))\|_{L^{2}}\\
&+\|\cfrac{1}{H^{d}}\int_{[0,H]^{d}}(P_{H_{\epsilon},H}^{z}-P_{H_{\epsilon}}^{z})(\Pi_{0,H_{\epsilon}}^{z}(u_{\epsilon}),\Pi_{1,H_{\epsilon}}^{z}(u_{\epsilon}))\|_{L^{2}}\\
  &+\|\cfrac{1}{H^{d}}\int_{[0,H]^{d}}P_{H_{\epsilon},H}^{z}(U_{0,H_{\epsilon},H}-\Pi_{0,H_{\epsilon}}^{z}(u_{\epsilon}),U_{0,H_{\epsilon},H}-\Pi_{0,H_{\epsilon}}^{z}(u_{\epsilon}))\|_{L^{2}}\\
\leq&  (CH_{\epsilon}^{2}+CH^{1+\alpha}+C\log(1/H_{\epsilon})H^{\alpha})\|f\|_{L^{2}}\leq CH^{\alpha}\|f\|_{L^{2}}.
\end{split}
\end{equation}

\end{proof}

We would like to make two important remarks.
\begin{itemize}

\item First, we note that the proof can be used for both high contrast 
and  bounded contrast cases. High contrast ($\zeta_\epsilon$) is assumed to be
with respect to numerical parameters, such as mesh size and RVE size. In high contrast cases, the resulting homogenized operator can contain high contrast
and, one needs to be careful with the constant $C_\zeta$ appearing in Assumption 1. In general, in high contrast case, our results show that the residual in
Lemma 4, is small. 

\item Secondly, we note that one can choose $H$ to be computational 
coarse mesh size that is
used to compute effective properties and effective properties can be piecewise 
constants on the coarse mesh.



\end{itemize}

\section{Numerical Approaches}
In this subsection, we will propose some numerical simplifications for the calculation of the effective operator.
In practice, instead of using (\ref{eq:upscale_operator}), we can consider the effective bilinear
form defined as
\begin{align*}
&\tilde{a}_{H_{\epsilon},H}^{z}((U_{0},U_{1}),(V_{0},V_{1}))\\
:=&\int_{\Omega}\Big(\sum_{i,j,k,l}\alpha_{kl,H_{\epsilon},H}^{ij,z}\partial_{k}U_{i}\partial_{l}V_{j}+\sum_{i,j,k}(\beta_{k,H_{\epsilon},H}^{ij,z}\partial_{k}U_{i}V_{j}+\beta_{k,H_{\epsilon},H}^{ji,z}\partial_{k}V_{j}U_{i})+\sum_{i,j}\gamma_{H_{\epsilon},H}^{ji,z}U_{i}V_{j}\Big)
\end{align*}
where 
\begin{align*}
\alpha_{kl,H_{\epsilon},H}^{ij,z}(\tilde{x}) & =\cfrac{1}{H^{d}}\int_{K_{x,H}(x)}\kappa(y)\nabla_{y}\eta_{x,i}^{(k)}(x,y)\cdot\nabla_{y}\eta_{x,j}^{(l)}(x,y)dy\;\forall\tilde{x}\in K_{x,H}(x),x\in I_{z,H},\\
\beta_{k,H_{\epsilon},H}^{ij,z}(\tilde{x}) & =\cfrac{1}{H^{d}}\int_{K_{x,H}(x)}\kappa(y)\nabla_{y}\eta_{x,i}^{(k)}(x,y)\cdot\nabla_{y}\eta_{x,j}(x,y)dy\;\forall\tilde{x}\in K_{x,H}(x),x\in I_{z,H},\\
\beta_{H_{\epsilon},H}^{ij,z}(\tilde{x}) & =\cfrac{1}{H^{d}}\int_{K_{x,H}(x)}\kappa(y)\nabla_{y}\eta_{x,i}(x,y)\cdot\nabla_{y}\eta_{x,j}(x,y)dy\;\forall\tilde{x}\in K_{x,H}(x),x\in I_{z,H}.
\end{align*}

We remark that the convergence analysis for this effective equation
is similar to the previous case under Assumption \ref{ass:regularity}.

Therefore, in each coarse element $K_{x,H}(x)$, we only need to compute a set of local solutions, $\eta^{(k)}_{x,i},\eta_{x,i}$, in oversampling domain $K^{+}_{x,H}(x)$. However, it can be still computationally expansive in some applications. We will then discuss a representative volume element (RVE) approach to reduce the computational cost.

\begin{assumption}
For all $x$, we assume there exist a $\omega(x)\subset K_{x,H}(x)$ such that
\[
\int_{K_{x,H}(x)}\kappa(y)\nabla_{y}\eta(x,y)\cdot\nabla_{y}\xi(x,y)dy \approx \cfrac{|K_{x,H}|}{|\omega(x)|}\int_{\omega(x)}\kappa(y)\nabla_{y}\eta(x,y)\cdot\nabla_{y}\xi(x,y)dy
\]
for all $\eta, \xi = \eta^{(k)}_{x,i} \text{ or }\eta_{x,i}$.
\end{assumption}
Using this assumption, we can compute the local solutions, $\eta^{(k)}_{x,i},\eta_{x,i}$, in the oversampling domain $\omega^+(x)\subset K^{+}_{x,H}(x)$. The effective coefficients can then be defined as  
\begin{align*}
\alpha_{kl,H_{\epsilon},H}^{ij,RVE}(\tilde{x}) & =\cfrac{|K_{x,H}|}{H^{d}|\omega(x)|}\int_{\omega(x)}\kappa(y)\nabla_{y}\eta_{x,i}^{(k)}(x,y)\cdot\nabla_{y}\eta_{x,j}^{(l)}(x,y)dy\;\forall\tilde{x}\in K_{x,H}(x),x\in I_{z,H},\\
\beta_{k,H_{\epsilon},H}^{ij,RVE}(\tilde{x}) & =\cfrac{|K_{x,H}|}{H^{d}|\omega(x)|}\int_{\omega(x)}\kappa(y)\nabla_{y}\eta_{x,i}^{(k)}(x,y)\cdot\nabla_{y}\eta_{x,j}(x,y)dy\;\forall\tilde{x}\in K_{x,H}(x),x\in I_{z,H},\\
\beta_{H_{\epsilon},H}^{ij,RVE}(\tilde{x}) & =\cfrac{|K_{x,H}|}{H^{d}|\omega(x)|}\int_{\omega(x)}\kappa(y)\nabla_{y}\eta_{x,i}(x,y)\cdot\nabla_{y}\eta_{x,j}(x,y)dy\;\forall\tilde{x}\in K_{x,H}(x),x\in I_{z,H}.
\end{align*}
We remark that using this assumption, we can consider $H\ll H_\epsilon$ and view $\omega(x)^{+}$ as the oversampling domain with size $O(\log(H_\epsilon)H)$
\section{Conclusion}
In this paper, we have examined a multicontinuum homogenization technique for addressing flow equations in high-contrast media. The method involves the creation of local basis functions on a small scale $H_{\epsilon}$, with a reference point $z$, following the concept of Constraint Energy Minimizing method (CEM). These basis functions enable the definition of a CEM downscaling operator. The multicontinuum homogenization is subsequently achieved by utilizing the CEM downscaling operator and taking a local average with respect to $z$ on a scale $H$, where $H_{\epsilon}\ll H\ll 1$.

Assuming the regularity of the local average of the solution, the analysis is based on the approximation property of the CEM downscaling operator and the error estimate between the CEM downscaling operator and the multicontinuum homogenization downscaling operator. We demonstrate that the solution's residual converges to zero at a certain order, thereby validating the convergence of the proposed homogenization method.

\section{Acknowledgement}
The research of Wing Tat Leung is partially supported by the Hong Kong RGC Early Career Scheme (Project number 9048274).
\bibliographystyle{abbrv}
\bibliography{references,references4,references1,references2,references3,decSol}

\end{document}